\documentclass[11pt,a4paper]{article}

\usepackage{epsf,epsfig,amsfonts,amsgen,amsmath,amstext,amsbsy,amsopn,amsthm,cases,listings,color
}
\usepackage{ebezier,eepic}
\usepackage{color}
\usepackage{multirow}
\usepackage{epstopdf}
\usepackage{graphicx}
\usepackage{pgf,tikz}
\usepackage{mathrsfs}
\usepackage[marginal]{footmisc}
\usepackage{enumitem}
\usepackage[titletoc]{appendix}
\usepackage{booktabs}
\usepackage{url}
\usepackage{mathtools}

\usepackage{pgfplots}
\usepackage{authblk}
\usepackage{amssymb}
\usepackage{float}
\usepackage{subfigure}
\pgfplotsset{compat=1.18}
\usepackage{mathrsfs}
\usetikzlibrary{arrows}

 \usepackage[backref=page]{hyperref}

\usepackage{algorithm}
\usepackage{algpseudocode}

\usepackage{booktabs,tabularx,ragged2e}
\newcommand{\problemStatement}[3]{%
  \begin{center}
  \begin{tabularx}{\columnwidth}{@{}lX@{}}
  \toprule
  \multicolumn{2}{@{}l@{}}{\textsc{#1}}\tabularnewline
  \midrule
  \bfseries Input:    & #2 \\
  \bfseries Question: & #3 \\
  \bottomrule
  \end{tabularx}
  \end{center}
}

\allowdisplaybreaks[1]

\definecolor{uuuuuu}{rgb}{0.27,0.27,0.27}
\definecolor{sqsqsq}{rgb}{0.1255,0.1255,0.1255}

\newtheorem{definition}{Definition} [section]
\newtheorem{theorem}[definition]{Theorem}

\newtheorem{claim}[definition]{Claim}

\newenvironment{pf}{\noindent{\bf Proof}}

\begin{document}
\title{\bf\Large A criterion for Andr\'{a}sfai--Erd\H{o}s--S\'{o}s type theorems and applications}

\date{\today}
\author[1]{Jianfeng Hou\thanks{Research was supported by National Key R\&D Program of China (Grant No. 2023YFA1010202), National Natural Science Foundation of China (Grant No. 12071077), the Central Guidance on Local Science and Technology Development Fund of Fujian Province (Grant No. 2023L3003). Email: \texttt{jfhou@fzu.edu.cn}}}
\author[2]{Xizhi Liu\thanks{Research was supported by ERC Advanced Grant 101020255. Email: \texttt{xizhi.liu.ac@gmail.com}}}
\author[1]{Hongbin Zhao\thanks{Email: \texttt{hbzhao2024@163.com}}}
\affil[1]{Center for Discrete Mathematics,
            Fuzhou University, Fujian, 350003, China}
\affil[2]{Mathematics Institute and DIMAP,
            University of Warwick,
            Coventry, CV4 7AL, UK}
\maketitle
\begin{abstract}
    The classical Andr\'{a}sfai--Erd\H{o}s--S\'{o}s Theorem states that for $\ell\ge 2$, every $n$-vertex $K_{\ell+1}$-free graph with minimum degree greater than $\frac{3\ell-4}{3\ell-1}n$ must be $\ell$-partite. 
    We establish a simple criterion for $r$-graphs, $r \geq 2$, to exhibit an Andr\'{a}sfai--Erd\H{o}s--S\'{o}s type property, also known as degree-stability. This leads to a classification of most previously studied hypergraph families with this property
    An immediate application of this result, combined with a general theorem by Keevash--Lenz--Mubayi, solves the spectral Tur\'{a}n problems for a large class of hypergraphs.

    We show an interesting application of the degree-stability in Complexity Theory: For every $r$-graph $F$ with degree-stability, there is a simple algorithm to decide the $F$-freeness of an $n$-vertex $r$-graph with minimum degree greater than $(\pi(F) - \varepsilon_F)\binom{n}{r-1}$ in time $O(n^r)$, where $\varepsilon_F >0$ is a constant. 
    In particular, for the complete graph $K_{\ell+1}$, we can take $\varepsilon_{K_{\ell+1}} = (3\ell^2-\ell)^{-1}$, and this bound is tight up to some multiplicative constant factor unless $\mathbf{W[1]} = \mathbf{FPT}$. 
    Based on a result by Chen--Huang--Kanj--Xia, we further show that for every fixed $C > 0$, this problem cannot be solved in time $n^{o(\ell)}$ if we replace $\varepsilon_{K_{\ell+1}}$ with $(C\ell)^{-1}$ unless $\mathbf{ETH}$ fails.
    Furthermore, we apply the degree-stability of $K_{\ell+1}$ to decide the $K_{\ell+1}$-freeness of graphs whose size is close to the Tur\'{a}n bound in time $(\ell+1)n^2$, partially improving a recent result by Fomin--Golovach--Sagunov--Simonov.
    
    As an intermediate step, we show that for a specific class of $r$-graphs $F$, the (surjective) $F$-coloring problem can be solved in time $O(n^r)$, provided the input $r$-graph has $n$ vertices and a large minimum degree, refining several previous results. 

\medskip


\end{abstract}
\section{Introduction}
Fix an integer $r\ge 2$, an $r$-graph $\mathcal{H}$ is a collection of $r$-subsets of some finite set $V$. We identify a hypergraph\footnote{A graph is viewed as a $2$-uniform hypergraph.} $\mathcal{H}$ with its edge set and use $V(\mathcal{H})$ to denote its vertex set. The size of $V(\mathcal{H})$ is denoted by $v(\mathcal{H})$. 
For a vertex $v\in V(\mathcal{H})$, 
the \textbf{degree} $d_{\mathcal{H}}(v)$ of $v$ in $\mathcal{H}$ is the number of edges in $\mathcal{H}$ containing $v$.
We use $\delta(\mathcal{H})$, $\Delta(\mathcal{H})$, and $d(\mathcal{H})$ to denote the \textbf{minimum degree}, the \textbf{maximum degree}, and the \textbf{average degree} of $\mathcal{H}$, respectively.
We will omit the subscript $\mathcal{H}$ if it is clear from the context.

Given a family $\mathcal{F}$ of $r$-graphs, we say $\mathcal{H}$ is \textbf{$\mathcal{F}$-free}
if it does not contain any member of $\mathcal{F}$ as a subgraph.
The \textbf{Tur\'{a}n number} $\mathrm{ex}(n,\mathcal{F})$ of $\mathcal{F}$ is the maximum
number of edges in an $\mathcal{F}$-free $r$-graph on $n$ vertices.
The \textbf{Tur\'{a}n density} of $\mathcal{F}$ is defined as $\pi(\mathcal{F}) \coloneqq \lim_{n\to \infty}\mathrm{ex}(n,\mathcal{F})/\binom{n}{r}$,
the existence of the limit follows from a simple averaging argument of Katona--Nemetz--Simonovits~\cite{KNS64}. 
We say a family $\mathcal{F}$ is \textbf{nondegenerate} if $\pi(\mathcal{F}) > 0$.
The study of $\mathrm{ex}(n,\mathcal{F})$ has been a central topic in extremal graph and hypergraph theory since the seminal work of Tur\'{a}n~\cite{T41} who proved that $\mathrm{ex}(n,K_{\ell+1}) = \left\lfloor \frac{\ell-1}{2\ell} n^2 \right \rfloor$ for all $n \ge \ell \ge 2$ (with the case $\ell=2$ proved even earlier by Mantel~\cite{Mantel07}). 
We refer the reader to surveys~\cite{Fur91,Sid95,Keevash11} for related results. 

\subsection{Degree-stability}\label{SUBSEC:degree-stable}
In this subsection, we focus on the structure of dense $F$-free hypergraphs. 
Here, dense can refer to either a large number of edges or a large minimum degree. 
This important topic in Extremal Combinatorics traces its origins to the seminal work of Simonovits~\cite{SI68}, who proved that for $\ell\ge 2$, every $K_{\ell+1}$-free $n$-vertex graph with at least $\frac{\ell-1}{2\ell} n^2 - o(n^2)$ edges can be made $\ell$-partite by removing $o(n^2)$ edges (see~\cite{Fur15} for a concise proof). 
The classical Andr\'{a}sfai--Erd\H{o}s--S\'{o}s Theorem~\cite{AES74} demonstrates an even stronger stability for $K_{\ell+1}$-free graphs. 
It states that for $\ell\ge 2$, every $n$-vertex $K_{\ell+1}$-free graph with minimum degree greater than $\frac{3\ell-4}{3\ell-1}n$ must be $\ell$-partite. 

Let $r\ge 2$ be an integer, $\mathcal{F}$ be a nondegenerate family of $r$-graphs, and $\mathfrak{H}$ be a family of $\mathcal{F}$-free $r$-graphs. We say 
\begin{itemize}
    \item $\mathcal{F}$ is \textbf{edge-stable} with respect to $\mathfrak{H}$ if for every $\delta>0$ there exist $\varepsilon>0$ and $n_0$ such that every $\mathcal{F}$-free $r$-graph $\mathcal{H}$ on $n \ge n_0$ vertices with $|\mathcal{H}| \ge \left(\pi(F)/r! - \varepsilon\right)n^r$ becomes a member in $\mathfrak{H}$ after removing at most $\delta n^r$ edges, 
    \item $\mathcal{F}$ is \textbf{degree-stable} with respect to $\mathfrak{H}$ if there exist $\varepsilon>0$ and $n_0$ such that every $\mathcal{F}$-free $r$-graph $\mathcal{H}$ on $n \ge n_0$ vertices with $\delta(\mathcal{H}) \ge \left(\pi(F)/(r-1)! - \varepsilon\right)n^{r-1}$ is a member in $\mathfrak{H}$, 
    \item $\mathcal{F}$ is \textbf{vertex-extendable} with respect to $\mathfrak{H}$ if there exist $\varepsilon>0$ and $n_0$ such that for every $\mathcal{F}$-free $r$-graph $\mathcal{H}$ on $n \ge n_0$ vertices with $\delta(\mathcal{H}) \ge \left(\pi(F)/(r-1)! - \varepsilon\right)n^{r-1}$ the following holds: 
        if $\mathcal{H}-v$ is a member in $\mathfrak{H}$, then $\mathcal{H}$ is a member in $\mathfrak{H}$ as well. 
\end{itemize}

Vertex-extendability was introduced in~\cite{LMR23unif} to provide a unified framework for proving the degree-stability of certain classes of nondegenerate families of hypergraphs. However, one limitation of the main results in~\cite{LMR23unif} is that they only apply to families $\mathcal{F}$ that are either blowup-invariant (see~{\cite[Theorem 1.7]{LMR23unif}})
or have a strong stability called vertex-stability (see~{\cite[Theorem 1.8]{LMR23unif}}).
In the following theorem, we extend the results of~\cite{LMR23unif} to include the broader class of families with edge-stability, a property satisfied by almost all nondegenerate hypergraph families (with known Tur\'{a}n densities).
\begin{theorem}\label{THM:edge-stable-to-degree-stable}
    Let $\mathcal{F}$ be a nondegenerate family of $r$-graph and $\mathfrak{H}$ be a hereditary\footnote{A family $\mathfrak{H}$ is hereditary if all subgraphs of every member $H\in \mathfrak{H}$ are also contained in $\mathfrak{H}$.} class of $\mathcal{F}$-free $r$-graphs. 
    If $\mathcal{F}$ is both edge-stable and vertex-extendable with respect to $\mathfrak{H}$, then $\mathcal{F}$ is degree-stable with respect to $\mathfrak{H}$. 
\end{theorem}
Recall that an \textbf{$r$-multiset} is an unordered collection of $r$ elements with repetitions allowed.
The \textbf{multiplicity $e(i)$} of $i$ in a multiset $e$ is the number of times that $i$ appears.
An \textbf{$r$-pattern} is a pair $P=(\ell,E)$ where $\ell$ is a positive integer and 
$E$ is a collection of $r$-multisets on $[\ell]$. 
It is clear that pattern is a generalization of $r$-graph, 
since an $r$-graph is a pattern in which $E$ consists of only simple $r$-sets. 
For convenience, we call $E$ the edge set of $P$ and omit the first coordinate $\ell$ if it is clear from the context. 
An $r$-graph $\mathcal{H}$ is \textbf{$P$-colorable} if there exists a \textbf{homomorphism} $\phi$ from $\mathcal{H}$ to $P$, where homomorphism means $\phi(e) \in P$ for every $e\in \mathcal{H}$. 

For many $r$-graph families $\mathcal{F}$, extremal $\mathcal{F}$-free constructions are typically $P$-colorable for some specific pattern $P$. 
We refer to such a pair $(\mathcal{F}, P)$ as a \textbf{Tur\'{a}n pair}. 
More specifically, given a pair $(F,P)$, where $P$ is an $r$-uniform pattern and $F$ is a family of $r$-graphs, we say $(F,P)$ is a Tur\'{a}n pair if every $P$-colorable $r$-graph is $F$-free and every $n$-vertex $F$-free $r$-graph with $\mathrm{ex}(n,F)-o(n^r)$ edges is $P$-colorable after removing at most $o(n^r)$ edges.
In most cases, applying Theorem~\ref{THM:edge-stable-to-degree-stable} involves choosing $\mathfrak{H}$ as the collection of all $P$-colorable hypergraphs. 
Hence, when stating that $\mathcal{F}$ is edge-stable/degree-stable/vertex-extendable, it is implied that this property is with respect to the family of $P$-colorable hypergraphs for simplicity.

\begin{table}[H]
    \centering
    \begin{tabular}{c|c}
        \hline
        Hypergraph & Degree-stable?\\
        \hline
        Edge-critical graphs~\cite{AES74,ES73} & Yes \\
        Non-edge-critical graphs~\cite{LMR23unif} & No\\
        Expansion of edge-critical graphs~\cite{MU06,PI13,LIU19,LMR23unif} & Yes \\
        Expansion of non-edge-critical graphs & No \\
        Expansion of extended Erd\H{o}s--S\'{o}s tree~\cite{Sido89,NY18,BIJ17} & Yes \\
        Expansion of $M_{2}^r$ for $r \ge 3$~\cite{HK13,BNY19,LMR23unif} & Yes \\
        Expansion of $M_{k}^3$, $L_{k}^3$, or $L_{k}^4$ for $k \ge 2$~\cite{HK13,JPW18,LMR23unif} & Yes\\
        Expansion of $M_{k}^{4}$ for $k \ge 2$~\cite{YP23} & Yes \\
        Expansion of $K_{4}^{3} \sqcup K_{3}^{3}$~\cite{YP22} & Yes \\
        Generalized triangle $\mathbb{T}_r$ for $r\in \{3,4\}$~\cite{BO74,FF83,KM04,LIU19,LMR23unif,Sido87,PI08} & Yes \\
        Generalized triangle $\mathbb{T}_r$ for $r\in \{5,6\}$~\cite{FF89,PI08} & No\\
        Expanded triangle $\mathcal{C}_{3}^{2r}$~\cite{Frankl90,KS05a} & No\\
        Fano Plane~\cite{DF00,KS05,FS05} & Yes \\
        $\mathbb{F}_{3,2}$ ($3$-book with $3$ pages)~\cite{FPS053Book3page} & No \\
        $F_7$ ($4$-book with $3$ pages)~\cite{FPS06Book} & Yes\\
        $\mathbb{F}_{4,3}$ ($4$-book with $4$ pages)~\cite{FMP084Book4page} & No \\
        \hline
    \end{tabular}
    \caption{List of hypergraphs with or without degree-stability.}
    \label{tab:degree-stable.}
\end{table}

In Table~\ref{tab:degree-stable.}, we summarize (most of) the previously studied hypergraph families (their definitions are included in the Appendix.) with degree-stability. Since they are all edge-stable, according to Theorem~\ref{THM:edge-stable-to-degree-stable}, proving degree-stability is reduced to verifying their vertex-extendability. 
This verification is relatively straightforward, and we refer the reader to~\cite{LMR23unif,HLLYZ23} for systematic results on this property.

Proof for Theorem~\ref{THM:edge-stable-to-degree-stable} is presented in Section~\ref{SEC:proof-degree-stable}.

\subsection{Spectral Tur\'{a}n problems}\label{SUBSEC:Intro-spectral}
In this subsection, we show a quick application of Theorem~\ref{THM:edge-stable-to-degree-stable} in spectral Tur\'{a}n problems. 

Given an $r$-graph $\mathcal{H}$ on $[n]$, the \textbf{Lagrange polynomial} of $\mathcal{H}$ is defined as 
\begin{align*}
    \Lambda_{\mathcal{H}}(X_1, \ldots, X_{n})
    \coloneqq 
    \sum_{E\in \mathcal{H}} \prod_{i\in E} X_i. 
\end{align*}
For every real number $\alpha \ge 1$, the \textbf{$\alpha$-spectral radius} of $\mathcal{H}$ is 
\begin{align*}
    \lambda_{\mathcal{H},\alpha}
    \coloneqq \max\left\{\Lambda_{\mathcal{H}}(x_1, \ldots, x_{n}) \colon (x_1, \ldots, x_{n}) \in \mathbb{S}^{n-1}_{\alpha}\right\}, 
\end{align*}
where $\mathbb{S}^{n-1}_{\alpha} \coloneqq \left\{(x_1, \ldots, x_{n}) \in \mathbb{R}^{n} \colon x_1^{\alpha} + \cdots + x_{n}^{\alpha} = 1\right\}$. 
%
The spectral Tur\'{a}n problem studies, for fixed family $\mathcal{F}$ of $r$-graphs, the value 
\begin{align*}
    \mathrm{specex}_{\alpha}(n,\mathcal{F})
    \coloneqq \max\left\{\lambda_{\mathcal{G},\alpha} \colon \text{$\mathcal{G}$ is an $n$-vertex $\mathcal{F}$-free $r$-graph}\right\}. 
\end{align*}
This problem was the focus of long-term research by Nikiforov (see the survey~\cite{Niki14}) and is relatively well-understood for graphs, with the general theorem by Wang--Kang--Xue~\cite{WKX23} being a notable highlight. 
However, similar to the ordinary Tur\'{a}n problem, understanding of the hypergraph spectral Tur\'{a}n problem remains relatively limited, with only a few known examples. 
By combining Theorem~\ref{THM:edge-stable-to-degree-stable} with a general theorem of Keevash--Lenz--Mubayi~\cite{KLM14}, we derive the following theorem, significantly enhancing our understanding of the hypergraph spectral Tur\'{a}n problem. 

\begin{theorem}\label{THM:spectral-Turan}
    Suppose that $F$ is an $r$-graph in Table~\ref{tab:degree-stable.} with degree-stability. 
    Then for sufficiently large $n$, every $F$-free $r$-graph $\mathcal{H}$ on $n$ vertices satisfies 
    \begin{align*}
        \lambda_{\mathcal{H}, \alpha}
        \le \max\left\{\lambda_{\mathcal{G}, \alpha} \colon \text{$\mathcal{G}$ is $P$-colorable and $v(\mathcal{H}) = n$}\right\}, 
    \end{align*}
    where $P$ is the pattern such that $(F,P)$ is a Tur\'{a}n pair. 
    Moreover, equality holds only if $\mathcal{H}$ is $P$-colorable. 
\end{theorem}

Theorem~\ref{THM:spectral-Turan} follows relatively straightforwardly from the following theorem by Keevash--Lenz--Mubayi~\cite{KLM14}. 
For convenience, we use $x = y \pm \delta$ to represent that $y -\delta \le x \le y + \delta$. 
Given a family $\mathfrak{H}$ of $r$-graphs, we let 
\begin{align*}
    \lambda_{\alpha}(\mathfrak{H}, n)
    \coloneqq \max\left\{\lambda_{\mathcal{G}, \alpha} \colon \text{$\mathcal{G} \in \mathfrak{H}$ and $v(\mathcal{G}) = n$}\right\}. 
\end{align*}

\begin{theorem}[Keevash--Lenz--Mubayi~\cite{KLM14}]\label{THM:KLM14}
    Let $N \ge r \ge 2$, $\alpha > 1$, $\varepsilon > 0$, $\mathcal{F}$ be a nondegerate family of $r$-graphs, and $\mathfrak{H}$ be a hereditary family of $\mathcal{F}$-free $r$-graphs. 
    There exist $\delta > 0$ and $n_0 > N$ such that the following holds. 
    Suppose that 
    \begin{enumerate}
        \item\label{THM:KLM14-1} $\mathcal{F}$ is degree-stable with respect to $\mathfrak{H}$, 
        \item\label{THM:KLM14-2} $\mathrm{ex}(n,\mathcal{F}) - \mathrm{ex}(n,\mathcal{F})  = \pi(\mathcal{F}) \binom{n}{r-1} \pm \delta n^{r-1}$ for all $n \ge N$, and 
        \item\label{THM:KLM14-3} $\lambda_{\alpha}(\mathfrak{H}, n) = \frac{r! \cdot \mathrm{ex}(n,\mathcal{F})}{n^{r/\alpha}} \pm \delta n^{r-r/\alpha -1}$ for all $n \ge N$. 
    \end{enumerate}
    Then $\mathrm{specex}_{\alpha}(n,\mathcal{F})
        = \lambda_{\alpha}(\mathfrak{H}, n)$ for every $n \ge n_0$. 
\end{theorem}
In the special case where $\mathfrak{H}$ is the collection of all $P$-colorable $r$-graphs for some pattern $P$, Assumption~\ref{THM:KLM14-2} in Theorem~\ref{THM:KLM14} is automatically true due to~{\cite[Theorem~1.9]{HLLYZ23}}. 
Therefore, in this case, one only needs to verify Assumption~\ref{THM:KLM14-2}, which involves calculating the maximum $\alpha$-spectral radius of $n$-vertex $P$-colorable $r$-graphs. 
This task is generally non-trivial.  
However, for $r$-graphs in Table~\ref{tab:degree-stable.} with degree-stability, Assumption~\ref{THM:KLM14-2} can be verified either using theorems from~\cite{KNY15} or through calculations similar to those in the proof of~{\cite[Corollary~1.6]{KLM14}}. Hence, we omit the proof for Theorem~\ref{THM:spectral-Turan}.

\subsection{Deciding the $F$-freeness of dense hypergraphs}
In this subsection, we present applications of the degree-stability in the decision problem of testing whether a hypergraph is $F$-free, a fundamental problem in Complexity Theory. 

Given two $r$-graphs $F$ and $\mathcal{H}$, a map $\psi \colon  V(F) \to V(\mathcal{H})$ is an \textbf{embedding of $F$} if $\psi$ is injective and $\psi(e) \in \mathcal{H}$ for all $e \in F$. 
For a fixed $r$-graph $F$, let $\textsc{Embed}$-$F$ denote the following decision problem: 
\problemStatement{$F$-embedding}
{An $r$-graph $\mathcal{H}$ on $n$ vertices.}
{Is there an embedding from $F$ to $\mathcal{H}$?}
%
Notice that, for a fixed $r$-graph $F$, the brute-force search can solve $\textsc{Embed}$-$F$ in time $O(n^{v(F)})$. 
For $r=2$, faster algorithms improving the exponent $v(F)$ have been explored by many researchers~\cite{IR78,NP85,AYZ97,KKM00,EG04}. In the case of $r\geq 3$, Yuster~\cite{Yus06} demonstrated algorithms that improve the exponent $v(F)$ for certain classes of $r$-graphs.
However, an open problem remains: determining whether there are faster algorithms that improve the exponent $v(F)$ when $F$ is a complete $r$-graph with at least $r+1$ vertices. 
The current record for this problem is due to Nagle~\cite{Nag10}\footnote{Nagle's result was improved by a $\mathrm{polylog}(n)$ factor very recently~\cite{AFS24}.}.
On the other hand, for the complete graph $K_{\ell}$, a result by Chen--Huang--Kanj--Xia~\cite{CHKX06} shows that $\mathrm{Embed}$-$K_{\ell}$ cannot be solved in time $n^{o(\ell)}$ unless the $\mathbf{ETH}$ (exponential time hypothesis~\cite{IPZ01}) fails. 
For complete bipartite graphs $K_{t,t}$, 
a result by Lin~\cite{Lin18} shows that no algorithm can solve $\textsc{Embed}$-$K_{t,t}$ in time $n^{O(1)}$ if $\mathbf{W[1]} \neq \mathbf{FPT}$ (see~\cite{DF99}), and no algorithm can solve $\textsc{Embed}$-$K_{t,t}$ in time $n^{o(\sqrt{t})}$ if randomized $\mathbf{ETH}$ holds (see~\cite{CFK16}).

Motivated by the recent work of Fomin--Golovach--Sagunov--Simonov~\cite{FGSS23}, 
we consider the following two embedding problems in dense hypergraphs. 
For a fixed $r$-graph $F$, let $\textsc{Embed}_{\mathrm{avg}}$-$(F,n,k)$ denote the following decision problem: 
\problemStatement{$F$-embedding with average degree constraint}
{An $r$-graph $\mathcal{H}$ on $n$ vertices with $|\mathcal{H}| \ge \mathrm{ex}(n,F)-k$.}
{Is there an embedding from $F$ to $\mathcal{H}$?}
In addition, define  $\textsc{Embed}_{\mathrm{min}}$-$(F,\alpha)$ as the following decision problem: 
\problemStatement{$F$-embedding with minimum degree constraint}
{An $r$-graph $\mathcal{H}$ on $n$ vertices with $\delta(\mathcal{H}) \ge \left(\pi(F)- \alpha\right) \binom{n}{r-1}$.}
{Is there an embedding from $F$ to $\mathcal{H}$?}
Among other results, 
Fomin--Golovach--Sagunov--Simonov~\cite{FGSS23}  proved that $\textsc{Embed}_{\mathrm{avg}}$-$(K_{\ell},n,k)$ can be solved in time $2.49^k n^{O(1)}$. 
A key ingredient in their proof involves reducing $\textsc{Embed}_{\mathrm{avg}}$-$(K_{\ell},n,k)$ to the task of verifying the $K_{\ell}$-freeness of a graph with at most $5k$ vertices. 
This reduction is rooted in Erd\H{o}s' proof~\cite{Erd70}, which essentially employs the Zykov symmetrization~\cite{Zyk49} to establish the Tur\'{a}n theorem.
Regrettably, Erd\H{o}s' proof does not appear to readily extend to general graphs and hypergraphs, and as a consequence, neither does the reduction employed by Fomin--Golovach--Sagunov--Simonov. 

Using a different strategy based on the degree-stability, we show that  $\textsc{Embed}_{\mathrm{min}}$-$(F,\alpha)$ (and hence, $\textsc{Embed}_{\mathrm{avg}}$-$\left(F,n,\alpha\binom{n}{r-1}\right)$) can be solved in time $O(n^r)$ when $\alpha$ is sufficiently small. 
%
\begin{theorem}\label{THM:Embed-min-F}
    Suppose that $F$ is an $r$-graph in Table~\ref{tab:degree-stable.} with degree-stability. 
    Then there exist constants $\varepsilon_{F}>0$ and $C_{F}>0$ depending only on $F$ such that the problem $\mathrm{Embed_{min}}$-$(F,\alpha)$ can be solved in time $C_F n^r$ for all $\alpha \le \varepsilon_{F}$, where $n$ is the number of vertices of the input hypergraph.
\end{theorem}
A natural and interesting problem arises in determining the optimal upper bound for $\varepsilon_{F}$ in Theorem~\ref{THM:Embed-min-F}.
For $F = K_{\ell+1}$, we can establish the lower bound $\frac{1}{3\ell^2-\ell}$ for $\varepsilon_{K_{\ell+1}}$.
Moreover, this bound is tight up to some multiplicative constant factor if $\mathbf{W[1]} \neq \mathbf{FPT}$. 
In addition, 
using the results of Chen--Huang--Kanj--Xia~\cite{CHKX06}, we can also show that $\varepsilon_{K_{\ell+1}}$ must be of order $o\left(\frac{1}{\ell}\right)$, even with a running time of $n^{o(\ell)}$.
\begin{theorem}\label{THM:Embed-min-Klique}
    There is an algorithm that solves $\textsc{Embed}_{\mathrm{min}}$-$(K_{\ell+1},\alpha)$ in time $(\ell+1) n^2$ for all $\ell \ge 2$ and $\alpha < \frac{1}{3\ell^2-\ell}$, where $n$ is the number of vertices of the input graph. 
    On the other hand, for every $i \ge 2$ there exists $\delta_i >0$ and $n_i$ such that $\textsc{Embed}_{\mathrm{min}}$-$(K_{\ell+1},(\delta_i\ell^2)^{-1})$ cannot be solved in time $O(n^{i})$ for any $\ell \ge n_i$ if $\mathbf{W[1]} \neq \mathbf{FPT}$, and for every fixed $C>0$ $\textsc{Embed}_{\mathrm{min}}$-$(K_{\ell+1},(C\ell)^{-1})$ cannot be solved in time $n^{o(\ell)}$ if $\mathbf{ETH}$ holds. 
\end{theorem}
Another interesting problem is to characterize the family of nondegenerate $r$-graphs $F$ for which there exists a constant $\varepsilon_F>0$ such that $\textsc{Embed}_{\mathrm{min}}$-$(K_{\ell+1},\varepsilon_F)$ can be solved in time $O(n^r)$. 
The following result, whose proof relies on results by Lin~\cite{Lin18}, provides a natural class of graphs that do not lie in this family. 

Given a graph $F$, the \textbf{$t$-blowup} $F[t]$ (of $F$) is obtained from $F$ by replacing each vertex with a set of size $t$ and each edge with a corresponding complete bipartite graph.
It is well-known that $\pi(F) = \pi(F[t])$ for every $t \ge 1$ (see~\cite{ES83}). 
\begin{theorem}\label{THM:Embed-min-F-blowup-lower-bound}
    For every fixed $\ell \ge 2$, 
    there is no algorithm that solves $\textsc{Embed}_{\mathrm{min}}$-$(K_{\ell+1}[t],0)$ in time $n^{O(1)}$ if $\mathbf{W[1]} \neq \mathbf{FPT}$, and there is no algorithm that solves $\textsc{Embed}_{\mathrm{min}}$-$(K_{\ell+1}[t],0)$ in time $n^{o(\sqrt{t})}$ if randomized $\mathbf{ETH}$ holds. 
\end{theorem}

We would like to remind the reader that in  Theorem~\ref{THM:Embed-min-F}, one should view $F$ as fixed and $n$ as large (since for small $n$, we can just use the brute-force search).  
Similarly, in the second part of Theorem~\ref{THM:Embed-min-Klique}, the integer $\ell$ should also be considered fixed. 
Meanwhile, in the first part of Theorem~\ref{THM:Embed-min-Klique}, we do not require $\ell$ to be fixed. 

For the problem $\textsc{Embed}_{\mathrm{avg}}$-$(K_{\ell+1},n,k)$, the following result improves upon the running time provided by Fomin--Golovach--Sagunov--Simonov within a specific range.
\begin{theorem}\label{THM:Embed-avg-Klique}
    There is an algorithm that solves $\textsc{Embed}_{\mathrm{avg}}$-$(K_{\ell+1},n,k)$ in time $(\ell+4) n^2$ for integers $n \ge \ell \ge 2$ and $k \ge 0$ satisfying $\max\{6\ell^2,\ 30k\ell\} \le n$. 
\end{theorem}
Note that if $\ell$ is fixed, then Theorem~\ref{THM:Embed-avg-Klique} implies that $\textsc{Embed}_{\mathrm{avg}}$-$(K_{\ell+1},n,k)$ can be solved in time $O(n^2)$ when $k \le \frac{n}{30\ell}$. 
The following result shows that this linear bound cannot be improved to $n^{1+\delta}$ for any constant $\delta>0$, thus leaving an open problem to determine the optimal bound for $k$ in Theorem~\ref{THM:Embed-avg-Klique}. 
\begin{theorem}\label{THM:Embed-avg-Klique-lower-bound}
     Unless $\mathbf{ETH}$ fails, for every fixed $\delta>0$ there is no algorithm that solves $\textsc{Embed}_{\mathrm{avg}}$-$(K_{\ell+1},n,n^{1+\delta})$ in time $n^{o(\ell)}$. 
\end{theorem}

Proofs for Theorems~\ref{THM:Embed-min-F},~\ref{THM:Embed-min-Klique},~\ref{THM:Embed-min-F-blowup-lower-bound},~\ref{THM:Embed-avg-Klique}, and~\ref{THM:Embed-avg-Klique-lower-bound} are presented in Section~\ref{SEC:proof-embed}. 
\subsection{Homomorphism and surjective homomorphism}
To establish the results in the previous subsection, we need to determine whether a given $r$-graph with a large minimum degree exhibits a specific structure. 
This is a particular instance of the homomorphism (or coloring) problem. To address this objective, we delve into the study of the homomorphism problem for dense hypergraphs in this subsection.

Given an $r$-pattern $P=(\ell,E)$, let the \textbf{Lagrange polynomial} of $P$ be
\begin{align*}
    \Lambda_{P}(X_1, \ldots, X_{\ell})
    \coloneqq \,\sum_{e\in E}\; \prod_{i=1}^{\ell}\;
        \frac{X_i^{e(i)}}{e(i)!}.
\end{align*}
The \textbf{Lagrangian} $\lambda_{P}$ of $P$ is defined as 
\begin{align*}
    \lambda_{P}
    \coloneqq \max\left\{\Lambda_{P}(x_1, \ldots, x_{\ell}) \colon (x_1, \ldots, x_{\ell}) \in \Delta^{\ell-1}\right\}, 
\end{align*}
where $\Delta^{\ell-1} \coloneqq \left\{(x_1, \ldots, x_{\ell}) \in \mathbb{R}^{\ell}_{\ge 0} \colon x_1+ \cdots+x_{\ell} = 1\right\}$ is the standard $(\ell-1)$-dimensional simplex. 
It is worth mentioning that the definition of $\lambda_{P}$ is equivalent to the definition of $\lambda_{P,1}$ from Section~\ref{SUBSEC:Intro-spectral} when $P$ is a hypergraph (see e.g.~\cite{CD12,FGH13}). 
Lagrangian is a crucial concept and has many applications in Extremal Combinatorics, and we refer the reader to~\cite{MS65,Sido87,FF89,PI14} for some examples.

We say a pattern $P = (\ell,E)$ is \textbf{minimal} if $\lambda_{P-i} < \lambda_{P}$ for all $i\in [\ell]$. 
Here $P-i$ denotes the pattern obtained from $P$ by removing $i$ and all edges containing $i$. 
Simple calculations show that a graph is minimal iff it is complete. 

Given an $r$-graph $\mathcal{H}$ and an $r$-pattern $P=(\ell, E)$, a map $\phi \colon V(\mathcal{H}) \to V(P)$ is a \textbf{homomorphism} (or a \textbf{$P$-coloring}) if $\phi(e) \in E$ for all $e \in \mathcal{H}$. 
We say $\mathcal{H}$ is \textbf{$P$-colorable} if there is a homomorphism from $\mathcal{H}$ to $P$. 
For a fixed $r$-pattern $P$, the \textbf{$P$-coloring problem} consists in deciding whether there exists a homomorphism of a given input $r$-graph $\mathcal{H}$ to $P$. 
When $P$ is a graph, the classical Hell--Ne\v{s}et\v{r}il Theorem~\cite{HN90} states that the $P$-coloring problem is in $\mathbf{P}$ if $P$ is bipartite and is $\mathbf{NP}$-complete otherwise. 
For $r\ge 3$, the $P$-coloring problem is already $\mathbf{NP}$-complete when $P = K_{r}^r$ (i.e. the $r$-graph with only one edge)~\cite{Szy10}. 

For an $r$-pattern $P$ and a real number $\alpha \in [0,1]$, let $\textsc{Hom}(P,\alpha)$ denote the following problem: 
\problemStatement{Homomorphism with minimum degree constraint}
  {An $r$-graph $\mathcal{H}$ on $n$ vertices with $\delta(\mathcal{H}) \ge \alpha n^{r-1}$.}
  {Is there a homomorphism from $\mathcal{H}$ to $P$?}
Edwards~\cite{Edw86} was the first to explore $\textsc{Hom}(K_{\ell},\alpha)$, proving that for every $\ell\ge 3$, the problem $\mathrm{Hom}(K_{\ell},\alpha)$ is in $\mathbf{P}$ if $\alpha > \frac{\ell-3}{\ell-2}$ and is $\mathbf{NP}$-complete otherwise. 
Subsequent extensions to general graphs and hypergraphs were considered in~\cite{DNST02,Szy12}, although many problems in this direction remain unresolved. 
Unfortunately, the exponents in the running time for algorithms provided in~\cite{Edw86,DNST02,Szy12} for large $\alpha$ are excessively large and depend on $F$, rendering them impractical for our purposes (Theorem~\ref{THM:Embed-min-F}). 
Therefore, we present the following theorem, which efficiently solves $\textsc{Hom}(P,\alpha)$ in time $O(n^r)$ when $P$ is a minimal $r$-pattern and $\alpha$ is close to $r\lambda_{P}$ (observe that, by Euler's homogeneous function theorem, if $\alpha>r\lambda_P$, then there is not such a homomorphism).
\begin{theorem}\label{THM:hypergraph-homomorphism}
    Suppose that $P$ is a minimal $r$-pattern on $\ell$ vertices. 
    Then there exist $\varepsilon_P>0$ and $n_P$ depending only on $P$ such that the problem $\textsc{Hom}(P,\alpha)$ can be solved in time $(\ell+1)n^r$ for all for $n\ge n_P$ and $\alpha \ge r\lambda_P - \varepsilon_P$, where $n$ is the number of vertices of the input $r$-graph. 
    Moreover, for every $r$-graph $\mathcal{H}$ on $n\ge n_P$ vertices with $\delta(\mathcal{H}) \ge (r \lambda_P - \varepsilon_P)n^{r-1}$, the homomorphism from $\mathcal{H}$ to $P$ is unique (up to the automorphism\footnote{An automorphism of $P$ is simply a bijective homomorphism from $P$ to $P$.} of $P$) if it exists. 
\end{theorem}
For complete graphs, the requirement that $n\ge n_P$ is not necessary.   
\begin{theorem}\label{THM:Kk-homomorphism}
    For every $\ell \ge 2$, the problem $\textsc{Hom}(K_{\ell},\alpha)$ can be solved in time $(\ell+1)n^2$ for all $\alpha > \frac{3\ell-4}{3\ell-1}$,  where $n$ is the number of vertices of the input graph. 
    Moreover, for every graph $G$ on $n$ vertices with $\delta(G) > \frac{3\ell-4}{3\ell-1} n$, the homomorphism from $G$ to $K_{\ell}$ is unique (up to the automorphism of $K_{\ell}$) if it exists. 
\end{theorem}
\textbf{Remark.}
It would be interesting to determine the infimum of $\alpha$ for which $\textsc{Hom}(K_{\ell},\alpha)$ can be solved in time $O(n^{2})$. 
It follows from Theorem~\ref{THM:Kk-homomorphism} and Edwards' result that this infimum lies in the interval $\left[\frac{\ell-3}{\ell-2},\ \frac{3\ell-4}{3\ell-1}\right]$ (assuming $\mathbf{P} \neq \mathbf{NP}$).

We take a step further by 
extending Theorem~\ref{THM:hypergraph-homomorphism} to an important variant of the homomorphism problem, specifically the surjective homomorphism problem. 
For a fixed $r$-pattern $P$, the \textbf{surjective $P$-coloring problem} involves determining whether a surjective homomorphism exists for a given input $r$-graph $\mathcal{H}$ to $P$. 
Unlike the $P$-coloring problem, a Hell--Ne\v{s}et\v{r}il-type theorem for the graph surjective homomorphism problem is still elusive (see e.g. ~\cite{BKM12,GLP12,GPS12,MP15,FGZ18,GJMPS19} for some related results). 

For an $r$-pattern $P$ and a real number $\alpha \in [0,1]$, let $\textsc{SHom}(P,\alpha)$ denote the following decision problem\footnote{A quick observation is that the results in~\cite{Edw86,DNST02,Szy12} concerning $\textsc{Hom}(P,\alpha)$ can be easily extended to $\textsc{SHom}(P,\alpha)$ with a minor modification to their original proofs.}: 
\problemStatement{Surjective homomorphism with minimum degree constraint}
  {An $r$-graph $\mathcal{H}$ on $n$ vertices with $\delta(\mathcal{H}) \ge \alpha n^{r-1}$.}
  {Is there a surjective homomorphism from $\mathcal{H}$ to $P$?}
Let us introduce some technical definitions before stating the result.  
For an $r$-pattern $P= (\ell, E)$, let $\Phi_P$ denote the maximum value of the following optimization problem: 
\begin{equation}\label{opt:Phi-H}
\begin{aligned}
    \max \quad & z \\
     \mathrm{s.t.} \quad & \exists (x_1, \ldots, x_{\ell}) \in \Delta^{\ell-1} \quad\text{with}\quad  
     \partial_{i} \Lambda_{P}(x_1, \ldots, x_{\ell}) \ge z \quad\text{for all}\quad i\in [\ell].  
\end{aligned}
\end{equation}
Here $\partial_{i} \Lambda_{P}$ denotes the partial derivative of $\Lambda_P$ with respect to the $i$-th variable.
For convenience, let $D_P$ be the collection of all optimal solutions to optimization problem~\eqref{opt:Phi-H}, and let 
\begin{align*}
    \phi_P
    \coloneqq \min_{(x_1, \ldots, x_{\ell})\in D_{P}}\min\{x_i \colon i\in [\ell]\}.  
\end{align*} 
We say $P$ is \textbf{rigid} if $\phi_P > 0$ and  $\partial_{i}\Lambda_{P}(x_1, \ldots, x_{\ell})  = \Phi_P$ holds for all $i\in [\ell]$ and for all $(x_1, \ldots, x_{\ell})\in D_{P}$.
Otherwise, we say $P$ is \textbf{non-rigid}. 
A quick observation, derived from the Lagrangian multiplier method,  asserts that every minimal pattern $P$ is rigid (but not vise versa, for example, $C_{k}$ is rigid for $k \ge 5$ but not minimal). 
Furthermore, for a minimal $r$-pattern $P$ we have $r\lambda_P = \Phi_{P}$. 
\begin{theorem}\label{THM:hypergraph-surjective-homomorphism}
    Suppose $P$ is a rigid $r$-pattern on $\ell$ vertices. 
    Then there exist $\varepsilon_P>0$ and $n_P$ depending only on $P$ such that the problem $\textsc{SHom}(P,\alpha)$ can be solved in time $(\ell+1)n^r$ for all for $n\ge n_P$ and $\alpha \ge \Phi_P - \varepsilon_P$, where $n$ is the number of vertices of the input $r$-graph. 
    Moreover, for every $r$-graph $\mathcal{H}$ on $n\ge n_P$ vertices with $\delta(\mathcal{H}) \ge (\Phi_P - \varepsilon_P)n^{r-1}$, the surjective homomorphism from $\mathcal{H}$ to $P$ is unique (up to the automorphism of $P$) if it exists. 
\end{theorem}

Proofs for Theorems~\ref{THM:hypergraph-homomorphism},~\ref{THM:Kk-homomorphism}, and~\ref{THM:hypergraph-surjective-homomorphism} are presented in Section~\ref{SEC:proof-coloring}. 

\section{Proof of Theorem~\ref{THM:edge-stable-to-degree-stable}}\label{SEC:proof-degree-stable}
    Given an $\mathcal{F}$-free $r$-graph $\mathcal{H}$ with large minimum degree, our objective is to show that $\mathcal{H}$ is contained in $\mathfrak{H}$. 
    The strategy in the following proof is to first use the edge-stability of $\mathcal{F}$ to find a large minimum degree subgraph (i.e. $\mathcal{H}_1[U]$) of $\mathcal{H}$ that is contained in $\mathfrak{H}$.
    Then we add the vertices in $V(\mathcal{H})$ back to $\mathcal{H}_1[U]$ one by one, where adding a vertex means adding all edges in $\mathcal{H}\setminus\mathcal{H}_1[U]$ containing this vertex.
    Using the vertex-extensibility, we will show that adding vertices preserves the containment in $\mathfrak{H}$, and hence, in the end, we obtain $\mathcal{H} \in \mathfrak{H}$. 
    
\begin{proof}[Proof of Theorem~\ref{THM:edge-stable-to-degree-stable}]
    Let $\mathcal{F}$ be a nondegenerate family of $r$-graphs and $\mathfrak{H}$ be a hereditary family of $\mathcal{F}$-free $r$-graphs. 
    Fix $\varepsilon_1 \ge \varepsilon_2 > 0$ sufficiently small and $n_0$ sufficiently large such that 
    \begin{enumerate}
        \item\label{assum:deg-stab-1}  every $\mathcal{F}$-free $r$-graph $\mathcal{H}$ on $n \ge n_0$ vertices with $|\mathcal{H}| \ge \left(\pi(F)/r! - \varepsilon_{2}\right)n^r$ becomes a member in $\mathfrak{H}$ after removing at most $\varepsilon_1 n^r/2$ edges, and  
        \item\label{assum:deg-stab-2}  every $\mathcal{F}$-free $r$-graph $\mathcal{H}$ on $n \ge n_0$ vertices with $\delta(\mathcal{H}) \ge \left(\pi(F)/(r-1)! - 2\varepsilon_1^{1/3}\right)n^{r-1}$ the following holds: 
        if $\mathcal{H}-v$ is a member in $\mathfrak{H}$, then $\mathcal{H}$ is a member in $\mathfrak{H}$ as well.
    \end{enumerate}
    Take $0 < \varepsilon_3 \le \varepsilon_2/2$. 
    Let $\mathcal{H}$ be an $\mathcal{F}$-free $r$-graph on $n \ge 2n_0$ vertices with $\delta(\mathcal{H}) \ge \left(\pi(\mathcal{F})/(r-1)! - \varepsilon_3\right)n^{r-1}$. 
    We aim to show that $\mathcal{H} \in \mathfrak{H}$.

    First, notice that 
    \begin{align*}
        |\mathcal{H}|
        \ge \frac{n}{r} \times \delta(\mathcal{H})
        \ge \frac{n}{r} \times \left(\frac{\pi(\mathcal{F})}{(r-1)!} - \varepsilon_3\right)n^{r-1} 
        \ge \left(\frac{\pi(\mathcal{F})}{r!} - \varepsilon_3\right)n^{r}. 
    \end{align*}
    So it follows from Assumption~\ref{assum:deg-stab-1} that  there exists a subgraph $\mathcal{H}_1 \subset \mathcal{H}$ such that $\mathcal{H}_1 \in \mathfrak{H}$ and 
    \begin{align*}
        |\mathcal{H}_1| 
        \ge |\mathcal{H}| - \frac{\varepsilon_1}{2} n^r
        \ge \left(\frac{\pi(\mathcal{F})}{r!} - \varepsilon_3\right)n^{r} -  \frac{\varepsilon_1}{2} n^r
        \ge \left(\frac{\pi(\mathcal{F})}{r!} - \varepsilon_1\right)n^{r}. 
    \end{align*}
    Let 
    \begin{align*}
        V \coloneq V(\mathcal{H}),\ 
        Z \coloneq \left\{v\in V \colon d_{\mathcal{H}_1}(v) < \left(\pi(F)/(r-1)! - \varepsilon_1^{1/3}\right)n^{r-1}\right\}
        \ \text{and}\ 
        U \coloneq V\setminus Z. 
    \end{align*}
    Simple calculations show that $|Z| \le \varepsilon_1^{1/3}n$ (see e.g. calculations in the proof of~{\cite[Lemma~4.2]{LMR23a}}), and hence, 
    \begin{align*}
        \delta(\mathcal{H}_1[U])
        \ge \left(\pi(F)/(r-1)! - \varepsilon_1^{1/3}\right)n^{r-1} - |Z|n^{r-2}
        \ge \left(\pi(F)/(r-1)! - 2\varepsilon_1^{1/3}\right)n^{r-1}. 
    \end{align*}
    Since $\mathfrak{H}$ is hereditary and $\mathcal{H}_1 \in \mathfrak{H}$, we have $\mathcal{H}_1[U] \in \mathfrak{H}$. 
    Let $v_1, \ldots, v_n$ be an ordering of vertices in $V$ such that $\{v_1, \ldots, v_{|U|}\} = U$. 
    For convenience, let $V_i \coloneq \{v_1, \ldots, v_i\}$ for $i\in [n]$. 
    Let $\mathcal{G}_0 \coloneqq \mathcal{H}_1[U]$, and for $i\in [n]$ let $\mathcal{G}_i \coloneqq \mathcal{G}_{i-1} \cup \{e\in \mathcal{H}[V_i] \colon v_i\in e\}$. 
    In particular, note that $V(\mathcal{G}_i) = U$ for $i\le z$, and $V(\mathcal{G}_i) = V_i$ for $i \ge z$. 
    We prove by induction on $i$ that $\mathcal{G}_i \in \mathfrak{H}$. 
    The base case $i=0$ is clear, so we may assume that $i\ge 1$. 
    Assume that we have shown that $G_{i-1} \in \mathfrak{H}$ for some $i\ge 1$, and we want to show that $\mathcal{G}_i \coloneqq \mathcal{G}_{i-1} \cup \{e\in \mathcal{H}[V_i] \colon v_i\in e\}$ is also contained in $\mathfrak{H}$. 
    
    If $v_i \in U$,  
    then it follows from $\mathcal{G}_{i} - v_i \subset \mathcal{G}_{i-1} \in \mathfrak{H}$,  
    \begin{align*}
        \delta(\mathcal{G}_i) 
        \ge \delta(\mathcal{H}_1[U]) 
        \ge \left(\pi(F)/(r-1)! - 2\varepsilon_1^{1/3}\right)n^{r-1}, 
    \end{align*}
    and Assumption~\ref{assum:deg-stab-2} that $\mathcal{G}_i \in \mathfrak{H}$. 
    
    If $v_i\in Z$, 
    then, similarly, it follows from $\mathcal{G}_{i} - v_i \subset \mathcal{G}_{i-1} \in \mathfrak{H}$,
    \begin{align*}
        \delta(\mathcal{G}_{i}) 
        \ge \delta(\mathcal{H}) - |Z| n^{r-2}
        \ge \left(\frac{\pi(\mathcal{F})}{(r-1)!} - \varepsilon_3\right)n^{r-1} - \varepsilon_1^{1/3}n^{r-1}
        \ge \left(\frac{\pi(\mathcal{F})}{(r-1)!} - 2\varepsilon_1^{1/3}\right)n^{r-1}, 
    \end{align*}
    and Assumption~\ref{assum:deg-stab-2} that $\mathcal{G}_i \in \mathfrak{H}$. This proves our claim. 
    Therefore, $\mathcal{H} = \mathcal{G}_n \in \mathfrak{H}$. 
\end{proof}

\section{Proofs for Theorems~\ref{THM:hypergraph-homomorphism},~\ref{THM:Kk-homomorphism}, and~\ref{THM:hypergraph-surjective-homomorphism}}\label{SEC:proof-coloring}
We prove Theorems~\ref{THM:hypergraph-homomorphism},~\ref{THM:Kk-homomorphism}, and~\ref{THM:hypergraph-surjective-homomorphism} in this section. 
The core of the proofs is a simple clustering algorithm based on the distance (defined below) between a pair of vertices.

Given an $n$-vertex $r$-graph $\mathcal{H}$ the \textbf{link} $L_{\mathcal{H}}(v)$ of a vertex $v\in V(\mathcal{H})$ is 
\begin{align*}
    L_\mathcal{H}(v)
    \coloneqq 
    \left\{A \in \binom{V(\mathcal{H})}{r-1} \colon A \cup \{v\}\in \mathcal{H}\right\}.  
\end{align*}
The (Hamming) \textbf{distance} between two vertices $u, v\in V(\mathcal{H})$ is $\mathrm{dist}_{\mathcal{H}}(u,v) \coloneqq |L_{\mathcal{H}}(u) \triangle L_{\mathcal{H}}(v)|$.
It is clear that $\mathrm{dist}_{\mathcal{H}}(u,v)$ can be calculated in time $n^{r-1}$. 
Observe that for a graph $G$, the value $\mathrm{dist}_{G}(u,v)$ is simply the Hamming distance of the row vectors corresponding to $u$ and $v$ in the adjacency matrix of $G$. 
The following clustering algorithm based on the distance of vertices will be crucial for proofs in this section. 

\begin{algorithm}
\renewcommand{\thealgorithm}{1}
\caption{\textsc{Hamming Clustering}}\label{ALGO:HammingClustering}
\begin{algorithmic}

\State \textbf{Input:}  
A triple $(\mathcal{H}, \ell, \delta)$, where $\mathcal{H}$ is an $n$-vertex $r$-graph, $\ell \ge 2$ is an integer, and $\delta \in [0,1]$ is a real number.

\State \textbf{Output:} 
A partition $V(\mathcal{H}) = V_1 \cup \cdots \cup V_{\ell}$. 
 
\State \textbf{Operations:}
            \begin{enumerate}
                \item Take a vertex $v_1 \in V(\mathcal{H})$, and let $W_1 \coloneqq \left\{u\in V(\mathcal{H}) \colon \mathrm{dist}_{\mathcal{H}}(u,v_1) \le \delta n^{r-1} \right\}$. 
                Suppose that we have defined $W_1, \ldots, W_i$ for some $i\in [\ell-1]$. 
                If $V(\mathcal{H}) \setminus (W_1 \cup \cdots \cup W_i) \neq \emptyset$, then take an arbitrary vertex $v_{i+1}$ from it, and let 
                \begin{align*}
                    W_{i+1} \coloneqq 
                    \begin{cases}
                        \left\{u\in V(\mathcal{H}) \colon \mathrm{dist}_{\mathcal{H}}(u,v_{i+1}) \le \delta n^{r-1} \right\} & \text{if}\quad i\neq \ell-1, \\
                        V(\mathcal{H}) \setminus (W_1 \cup \cdots \cup W_i) & \text{if}\quad i=\ell-1. 
                    \end{cases} 
                \end{align*}
                Otherwise, let $W_{i+1} = \cdots = W_{\ell} = \emptyset$.    
                \item Let $V_i \coloneqq W_i \setminus (W_{i+1}\cup \cdots \cup W_{\ell})$ for $i\in [\ell-1]$, and let $V_{\ell} \coloneqq W_{\ell}$. 
            \end{enumerate}
\end{algorithmic}
\end{algorithm}

\begin{proof}[Proof of Theorem~\ref{THM:Kk-homomorphism}]
        Let $n \ge \ell \ge 2$ be integers. 
        Suppose that $G$ is an $n$-vertex graph with $\delta(G) = \alpha n$, where $\alpha > \frac{3\ell-4}{3\ell-1}$ is a real number. 
        Let $\delta \coloneqq \frac{2}{3\ell-1}$. 
        Run Algorithm~\ref{ALGO:HammingClustering} with input $(G, \ell, \delta)$, and let $V_1 \cup \cdots \cup V_{\ell} = V(G)$ denotes the output partition. It is easy to see that the running time for this step is at most $\ell n^2$. 

        \begin{claim}\label{CLAIM:K-coloring-V-empty}
            The graph $G$ is $\ell$-partite iff $\bigcup_{i\in [\ell]}G[V_i] = \emptyset$.
        \end{claim}
        \begin{proof}
            Suppose that $G$ is $\ell$-partite. 
            Suppose that $U_1 \cup \cdots \cup U_{\ell} = V(G)$ is a partition with $\bigcup_{i\in [\ell]}G[U_i] = \emptyset$ (this partition is used only for the proof and is not used for the algorithm).
            It suffices to show that $\{V_1, \ldots, V_{\ell}\} = \{U_1, \ldots, U_{\ell}\}$. 

            Let $x_i \coloneqq |U_i|/n$ for $i\in [\ell]$.
            Since $G$ is $\ell$-partite, we obtain $1-x_i = \sum_{j\in [\ell]\setminus\{i\}}x_j \ge \delta(G)/n > \alpha$ for all $i\in [\ell]$.
            Consequently, $x_i < 1-\alpha$ for all $i\in [\ell]$, and hence, $1-x_i = \sum_{j\in [\ell]\setminus\{i\}}x_j  < (\ell-1)(1-\alpha)$ for all $i\in [\ell]$. 
            In summary, we have 
            \begin{align}\label{equ:size-xi}
                1 - (\ell-1)(1-\alpha)
                < x_i 
                < 1-\alpha
                \quad\text{for all}\quad 
                i\in [\ell]. 
            \end{align}
            Fix $i, j \in [\ell]$ with $i \neq j$. 
            Suppose that $u, u' \in U_i$ are two distinct vertices. 
            Then it follows from the Inclusion-Exclusion Principle and~\eqref{equ:size-xi} that 
            \begin{align*}
                \mathrm{dist}_{G}(u,u')
                & = d_{G}(u)+ d_{G}(u') - 2 |N_{G}(u) \cap N_{G}(u')| \\
                & \le d_{G}(u)+ d_{G}(u') - 2 \left(d_{G}(u)+ d_{G}(u') - (1-x_i)n\right) \\
                & = 2(1-x_i)n - (d_{G}(u)+ d_{G}(u')) \\
                & \le 2(1-x_i - \alpha)n
                \le 2\left(\ell-1-\ell\alpha\right)n
                < \delta n. 
            \end{align*}
            According to Algorithm~\ref{ALGO:HammingClustering},  $u, u' \subset V_{i^{\ast}}$ for some $i^{\ast} \in [\ell]$. 
            
            Suppose that $v \in U_i$ and $v'\in U_j$ are two distinct vertices.
            Then by~\eqref{equ:size-xi}, 
            \begin{align*}
                \mathrm{dist}_{G}(v,v')
                & \ge d_{G}(v) - \sum_{k\in [\ell]\setminus\{i,j\}}x_k n + d_{G}(v') - \sum_{k\in [\ell]\setminus\{i,j\}}x_k n \\
                & \ge 2\left(\alpha - (1 - x_i-x_j)\right)n \\
                & \ge 2\left(\alpha +2 (1-(\ell-1)(1-\alpha)) -1\right)n  \\
                & = 2\left((2\ell-1)\alpha - 2\ell+3\right)n
                > \delta n.  
            \end{align*}
            According to Algorithm~\ref{ALGO:HammingClustering}, $v$ and $v'$ do not belong to the same part $V_i$. 
            This proves that $\{V_1, \ldots, V_{\ell}\} = \{U_1, \ldots, U_{\ell}\}$, and hence, $\bigcup_{i\in [\ell]}G[V_i] = \bigcup_{i\in [\ell]}G[U_i] = \emptyset$. 
            Additionally, observe that this proof shows the uniqueness of the partition $U_1 \cup \cdots \cup U_{\ell} = V(G)$. 
        \end{proof}
       According to Claim~\ref{CLAIM:K-coloring-V-empty}, to check whether $G$ is $\ell$-partite we just need to check whether $\bigcup_{i\in [\ell]}G[V_i] = \emptyset$. 
       This can be accomplished in at most $n^2$ time. Therefore, the overall time complexity is at most $(\ell+1)n^2$.
\end{proof}

Next, we present the proof for Theorem~\ref{THM:hypergraph-surjective-homomorphism}, while Theorem~\ref{THM:hypergraph-homomorphism} follows from a similar argument and its proof is omitted.  
\begin{proof}[Proof of Theorem~\ref{THM:hypergraph-surjective-homomorphism}]
    Fix a rigid $r$-pattern $P$ on $\ell$ vertices.
    For simplicity, let us assume that the vertex set of $P$ is $[\ell]$. 
    For every $\beta > 0$, let 
    \begin{align*}
        D_{P,\beta}
        \coloneqq \left\{(y_1, \ldots, y_{\ell}) \in \Delta^{\ell-1}\colon \exists (x_1, \ldots, x_{\ell})\in D_{P} \text{ with } \max_{i\in [\ell]}|x_i-y_i|< \beta \right\}, 
    \end{align*}
    and let $\Phi_{P,\beta}$ denote the optimal value of the following optimization problem: 
    \begin{align*}
        \max \quad & y \\
         \mathrm{s.t.} \quad & \exists (x_1, \ldots, x_{\ell}) \in \Delta^{\ell-1}\setminus D_{P,\beta} \quad\text{with}\quad  
         \partial_{i}\Lambda_{P}(x_1, \ldots, x_{\ell}) \ge y \quad\text{for all}\quad i\in [\ell]. 
    \end{align*}
    It is easy to see from the definitions that $\Phi_{P,\beta} < \Phi_{P}$ for all $\beta>0$.
    Take 
    \begin{align*}
        \delta 
        \coloneq \frac{\left(\phi_{P}/2\right)^{r-1}}{(r-1)!},\ 
        \delta_{\ast} 
        \coloneqq \frac{\delta}{2^{r} \ell} < \frac{1}{\ell},
        \ \text{and}\ 
        \varepsilon_{P} 
        \coloneqq \min\left\{\frac{\Phi_{P} - \Phi_{P,\delta_{\ast}}}{2},\ \frac{\delta}{5}\right\}. 
    \end{align*}
    Let $n$ be sufficiently large, and 
    let $\mathcal{H}$ be an $n$-vertex $r$-graph with $\delta(\mathcal{H}) = \alpha n^{r-1} \ge (\Phi_{P} - \varepsilon_{P}) n^{r-1}$. 
    Run Algorithm~\ref{ALGO:HammingClustering} with input $(\mathcal{H}, \ell, \delta)$ and let $V_1 \cup \cdots \cup V_{\ell} = V(\mathcal{H})$ denote the output partition. It is easy to see that the running time for this step is at most $\ell n^r$. 
    \begin{claim}\label{CLAIM:color-patter-Vi-empty}
        There is a surjective homomorphism from $\mathcal{H}$ to $P$ iff $\bigcup_{i\in [\ell]}\mathcal{H}[V_i] = \emptyset$ and $V_i \neq \emptyset$ for all $i\in [\ell]$. 
    \end{claim}
    \begin{proof}
        Suppose that $\psi \colon V(\mathcal{H}) \to [\ell]$ is a surjective homomorphism from $\mathcal{H}$ to $P$. 
        Let $U_i \coloneqq \psi^{-1}(i) \subset V(\mathcal{H})$ and $y_i \coloneqq |U_i|/n$ for $i\in [\ell]$. 
        Notice that $y_i > 0$ for $i\in [\ell]$ and $U_1 \cup \cdots \cup U_{\ell}$ is a partition of $V(\mathcal{H})$.
    Observe from definitions that for every $i\in [\ell]$ we have 
    \begin{align*}
        \Phi_{P,\delta_{\ast}}
        < \Phi_{P} - \varepsilon_{P}
        \le \alpha 
        = \frac{\delta(\mathcal{H})}{n^{r-1}}
        \le \partial_{i}\Lambda_{P}(y_1, \ldots, y_{\ell}). 
     \end{align*}
    So, it follows from the definitions of $\varepsilon_{P}$ and $\Phi_{P,\delta_{\ast}}$ that there exists $(x_1, \ldots, x_{\ell})\in D_{P}$ such that $\max_{i\in [\ell]}|x_i-y_i|< \delta_{\ast}$. 
    Since for every vector $(z_1, \ldots, z_{\ell})\in \mathbb{R}^{\ell}$ with $\max_{k\in[\ell]}|z_k -x_k| \le \delta_{\ast}$, the inequality 
    \begin{align*}
        |\partial^{2}_{i,j}\Lambda_{P}(z_1, \ldots, z_{\ell})|
        \le \left(|z_1|+ \cdots + |z_{\ell}|\right)^{r-2}
        \le (1+\ell \delta_{\ast})^{r-2}
        \le 2^{r-2}. 
    \end{align*}
    holds for all $i,j \in [\ell]$,  
    it follows Taylor's theorem that for every $i\in [\ell]$, 
    \begin{align*}
        \left|\partial_{i}\Lambda_{P}(y_1, \ldots, y_{\ell}) - \partial_{i}\Lambda_{P}(x_1, \ldots, x_{\ell})\right| 
        \le 2^{r-2}\cdot\sum_{k\in [\ell]}|y_k-x_k|
        \le 2^{r-2} \ell \delta_{\ast}. 
    \end{align*}
    Therefore, $\partial_{i}\Lambda_{P}(y_1, \ldots, y_{\ell})
        \le \partial_{i}\Lambda_{P}(x_1, \ldots, x_{\ell}) + 2^{r-2} \ell \delta_{\ast}
        = \Phi_{P}+ 2^{r-2} \ell \delta_{\ast}$.  

    Fix distinct $i,j \in [\ell]$. 
    Suppose that $u,u' \in U_i$. 
    Similar to the proof of Theorem~\ref{THM:Kk-homomorphism}, it follows from the Inclusion-Exclusion Principle and the inequality above that 
    \begin{align*}
        \mathrm{dist}_{\mathcal{H}}(u,u')
        & = d_{\mathcal{H}}(u) + d_{\mathcal{H}}(u') - 2|L_{\mathcal{H}}(u) \cap L_{\mathcal{H}}(u')| \\
        & \le d_{\mathcal{H}}(u) + d_{\mathcal{H}}(u') - 2\left(d_{\mathcal{H}}(u) + d_{\mathcal{H}}(u') - \partial_{i}\Lambda_{P}(y_1, \ldots, y_{\ell}) \cdot n^{r-1}\right)\\
        & = 2\cdot \partial_{i}\Lambda_{P}(y_1, \ldots, y_{\ell})\cdot n^{r-1} -\left(d_{\mathcal{H}}(u) + d_{\mathcal{H}}(u') \right) \\
        & \le 2\left(\Phi_{P}+ 2^{r-2} \ell \delta_{\ast} - \alpha\right) n^{r-1} \\
        &= 2\left(2^{r-2} \ell \delta_{\ast} + \varepsilon_{F}\right) n^{r-1}
        < \delta n^{r-1}. 
    \end{align*}
    Suppose that $v\in U_i$ and $v'\in U_j$ are two distinct vertices.
    A simple but crucial observation is that a rigid pattern does not contain twin vertices, i.e. vertices with the same link. 
    Therefore, the two polynominals $\partial_{i}\Lambda_{P}(X_1, \ldots, X_{\ell})$ and $\partial_{j}\Lambda_{P}(X_1, \ldots, X_{\ell})$ are not identical. 
    Hence, 
    \begin{align*}
        \mathrm{dist}_{\mathcal{H}}(v,v') 
        \ge \frac{\left(\min_{i\in [\ell]} y_i\right)^{r-1}}{(r-1)!} \cdot n^{r-1}
        > \frac{\left(\phi_P - \delta_{\ast}\right)^{r-1}}{(r-1)!} \cdot n^{r-1}
        > \delta n^{r-1}. 
    \end{align*}
    It follows from the definition of Algorithm~\ref{ALGO:HammingClustering} that $\{U_1, \ldots, U_{\ell}\} = \{V_1, \ldots, V_{\ell}\}$, and hence, $\bigcup_{i\in [\ell]}\mathcal{H}[V_i] = \emptyset$ and $V_i \neq \emptyset$ for all $i\in [\ell]$. 
    \end{proof}

    According to Claim~\ref{CLAIM:color-patter-Vi-empty}, to check whether there is a surjective homomorphism from $\mathcal{H}$ to $P$ we just need to check whether $\bigcup_{i\in [\ell]}\mathcal{H}[V_i] = \emptyset$ and $V_i \neq \emptyset$ for all $i\in [\ell]$.  
    This can be accomplished in at most $n^r$ time. Therefore, the overall time complexity is at most $(\ell+1)n^r$.
\end{proof}

\section{Proofs for Theorems~\ref{THM:Embed-min-F},~\ref{THM:Embed-min-Klique},~\ref{THM:Embed-min-F-blowup-lower-bound},~\ref{THM:Embed-avg-Klique}, and~\ref{THM:Embed-avg-Klique-lower-bound}}\label{SEC:proof-embed}
We prove Theorems~\ref{THM:Embed-min-F},~\ref{THM:Embed-min-Klique},~\ref{THM:Embed-min-F-blowup-lower-bound},~\ref{THM:Embed-avg-Klique}, and~\ref{THM:Embed-avg-Klique-lower-bound} in this section. 

\begin{proof}[Proof of Theorem~\ref{THM:Embed-min-F}]
    Let $F$ be an $r$-graph in Table~\ref{tab:degree-stable.} with degree-stability. 
    Let $P$ be the minimal pattern such that $(F,P)$ is a Tur\'{a}n pair. 
    Simple calculations show that $\pi(F) = r! \lambda_{P}$. 
    Let $\ell$ denote the number of vertices in $P$. 
    Let $\varepsilon_P>0$ and $n_P$ be constants given by Theorem~\ref{THM:hypergraph-homomorphism}. 
    Take $\varepsilon_F \in (0,\varepsilon_P)$ to be sufficiently small and $n_0\ge n_P$ be sufficiently large such that:  
     every $F$-free $r$-graph $\mathcal{H}$ on $n \ge n_0$ vertices with $\delta(\mathcal{H}) \ge \left(\pi(F) - \varepsilon_F\right)\binom{n}{r-1}$ is $P$-colorable (this is guaranteed by the degree-stability of $F$). 
    Consider the following algorithm: 
    
\begin{algorithm}
\renewcommand{\thealgorithm}{2}
\caption{\textsc{Deciding $F$-freeness in large minimum degree hypergraphs}}\label{ALGO:F-freeness-min-deg}
\begin{algorithmic}

\State \textbf{Input:}  
An $n$-vertex $r$-graph $\mathcal{H}$ with $\delta(\mathcal{H}) \ge (\pi(F) - \varepsilon_{F})\binom{n}{r-1}$. 

\State \textbf{Output:} 
"Yes", if $\mathcal{H}$ is $F$-free; "No", otherwise.  
 
\State \textbf{Operations:}
                
            \begin{enumerate}
                \item If $n < n_0$, then use the brute-force search to check the $F$-freeness of $\mathcal{H}$ and return the answer. 
                \item If $n \ge n_0$, then run the Algorithm~\ref{ALGO:HammingClustering} with input $(\mathcal{H}, \ell, \varepsilon_{F})$. 
                Assume that $V_1 \cup \cdots \cup V_{\ell} = V(\mathcal{H})$ is the output partition. 
                Check whether $\mathcal{H}[V_i] = \emptyset$
                holds for all $i\in [\ell]$. 
                If it does hold for all $i\in [\ell]$, then return "Yes"; otherwise, return "No". 
            \end{enumerate}
\end{algorithmic}
\end{algorithm}

It is easy to see from Theorem~\ref{THM:hypergraph-homomorphism}  that the running time of Algorithm~\ref{ALGO:F-freeness-min-deg} is at most $\max\{n_0^{v(F)},\ (\ell+1)n^r\} = O(n^{r})$, proving Theorem~\ref{THM:Embed-min-F}. 
\end{proof}
Next, we prove Theorem~\ref{THM:Embed-min-Klique}. 
\begin{proof}[Proof of Theorem~\ref{THM:Embed-min-Klique}]
    The proof for the first part of Theorem~\ref{THM:Embed-min-Klique} is similar to the proof of Theorem~\ref{THM:Embed-min-F}, so we omit it here and focus on the second part of Theorem~\ref{THM:Embed-min-Klique}.

    First, we prove that $\textsc{Embed}_{\min}$-$(K_{\ell+1},(C\ell)^{-1})$ cannot be solved in time $n^{o(\ell)}$ for any fixed $C$, unless $\mathbf{ETH}$ fails. 
    Indeed, fix $C>0$ (we may assume that $C$ is an integer) and suppose to the contrary that there exists an algorithm $\mathcal{A}$ that solves $\textsc{Embed}_{\min}$-$(K_{\ell+1},\alpha)$ in time $n^{o(\ell)}$ for $\alpha = \frac{1}{C\ell}$.
    We claim that $\mathcal{A}$ can also solves $\textsc{Embed}$-$K_{\ell+1}$ in time $n^{o(\ell)}$, which would contradict the result by Chen--Huang--Kanj--Xia~\cite{CHKX06}. 
    Indeed, consider an arbitrary $n$-vertex graph $G$. 
    Let $V_0 \coloneqq V(G)$. 
    Let $\hat{G}$ be the graph obtained from $G$ by adding $q \coloneqq C\ell$ sets $V_1, \ldots, V_{q}$, each of size $n$, and adding new edges $\{u,v\}$ for all $(u,v) \in V_i \times V_j$ whenever $0 \le i < j \le q$.
    Let $N \coloneqq (q+1)n$ denote the number of vertices in $\hat{G}$ and let $L\coloneqq q+\ell = (C+1)\ell$. 
    Observe that $K_{\ell+1} \subset G$ iff $K_{L+1} \subset \hat{G}$. 
    Since 
    \begin{align}\label{equ:CHKX-construction}
        \delta(\hat{G}) 
        \ge q n 
        = \frac{q}{q+1}N 
        =  \left(\frac{L-1}{L} - \frac{\ell-1}{(q+1) L}\right)N 
        > \left(\frac{L-1}{L} - \frac{1}{C L}\right)N, 
    \end{align}
    by assumption, algorithm $\mathcal{A}$ can decide in time $N^{o(L)} = \left((C+1)\ell n\right)^{o\left((C+1)\ell\right)} =  n^{o(\ell)}$ whether $K_{L+1} \subset \hat{G}$, and equivalently, whether $K_{\ell+1} \subset G$, proving our claim.

    Now assume that $\mathbf{W[1]} \neq \mathbf{FPT}$. 
    For every $i \ge 2$ let $\ell_i$ be the smallest integer such that $\textsc{Embed}$-$K_{\ell_i+1}$ cannot be solved in time $O(n^{i})$, let $\delta_i \coloneq \frac{1}{2(\ell_i-1)}$ and $n_i \coloneq 2\ell_i$.
    We claim that $\textsc{Embed}_{\min}$-$(K_{\ell+1},(\delta_i\ell^2)^{-1})$ cannot be solved in time $O(n^{i})$ for any $\ell \ge n_i$. 
    Indeed, suppose to the contrary that there exist an $i_{\ast} \ge 2$ and an algorithm $\mathcal{A}_{i_{\ast}}$ that solves $\textsc{Embed}_{\min}$-$(K_{\ell_{\ast}+1},(\delta_{i_{\ast}}\ell^2)^{-1})$ in time $O(n^{i_{\ast}})$ for some $\ell_{\ast} \ge n_{i_{\ast}}$. 
    Consider an arbitrary $n$-vertex graph $G$. 
    Let $\hat{G}$ be the same construction as defined above by replacing $\ell$ with $\ell_{i_{\ast}}$ and $C$ with $\ell_{\ast}/\ell_{i_{\ast}} - 1 \ge 1$ (hence, $L = \ell_{\ast}$ and $L/2 \ge \ell_{i_{\ast}}$). 
    It follows from~\eqref{equ:CHKX-construction} that 
    \begin{align*}
        \delta(\hat{G}) 
        \ge \left(\frac{L-1}{L} - \frac{\ell_{i_{\ast}}-1}{(q+1) L}\right)N 
        > \left(\frac{L-1}{L} - \frac{\ell_{i_{\ast}}-1}{L^2/2}\right)N
        = \left(\frac{L-1}{L} - \frac{1}{\delta_{i_{\ast}} L^2}\right)N. 
    \end{align*}
    By assumption, algorithm $\mathcal{A}_{i_{\ast}}$ can decide in time $O\left(N^{i_{\ast}}\right) = O\left(((C+1)\ell_{i_{\ast}} n)^{i_{\ast}}\right) = O\left(n^{i_{\ast}}\right)$ whether $K_{L+1} \subset \hat{G}$, and equivalently, whether $K_{\ell_{i_{\ast}}+1} \subset G$, contradicting the definition of $\ell_{i_{\ast}}$. 
\end{proof}


Next, we present the proof of Theorem~\ref{THM:Embed-min-F-blowup-lower-bound}. 

\begin{proof}[Proof of Theorem~\ref{THM:Embed-min-F-blowup-lower-bound}]
    Fix $\ell\ge 0$, and suppose to the contrary that there is an algorithm $\mathcal{A}$ that solves $\textsc{Embed}_{\mathrm{min}}$-$(K_{\ell+1}[t],0)$ in time $n^{O(1)}$ (or $n^{o(\sqrt{t})}$). 
    Consider an arbitrary $n$-vertex graph $G$.
    We may assume that $n \ge t$. 
    Let $V_0 \coloneqq V(G)$. 
    Let $\hat{G}$ be the graph obtained from $G$ by adding $\ell-1$ sets $V_1, \ldots, V_{\ell-1}$, each of size $n$, and adding new edges $\{u,v\}$ for all $(u,v) \in V_i \times V_j$ whenever $0 \le i < j \le \ell-1$.
    Let $N \coloneqq \ell n$ denote the number of vertices in $\hat{G}$. 
    Observe that $K_{t,t} \subset G$ iff $K_{\ell+1}[t] \subset \hat{G}$. 
    Since 
    \begin{align*}
        \delta(\hat{G}) 
        \ge (\ell-1) n 
        = \frac{\ell-1}{\ell}N
        = \pi(K_{\ell+1}[t])N, 
    \end{align*}
    by assumption, algorithm $\mathcal{A}$ can decide in time $N^{O(1)} = \left(\ell  n\right)^{O(1)} =  n^{O(1)}$ (or $N^{o(\sqrt{t})} = n^{o(\sqrt{t})}$) whether $K_{\ell+1}[t] \subset \hat{G}$, and equivalently, whether $K_{t,t} \subset G$, contradicting the result by Lin. 
\end{proof}

Next, we prove Theorem~\ref{THM:Embed-avg-Klique}. 
In the proof, we will use the Andr\'{a}sfai--Erd{\H o}s--S\'{o}s Theorem.
\begin{theorem}[Andr\'{a}sfai--Erd{\H o}s--S\'{o}s~\cite{AES74}]\label{THM:AES}
For all integers $n \ge \ell \ge 2$, every $n$-vertex $K_{\ell+1}$-free graph $G$ with $\delta(G) > \frac{3\ell-4}{3\ell - 1}n$ is $\ell$-partite. 
\end{theorem}

\begin{algorithm}
\renewcommand{\thealgorithm}{3}
\caption{\textsc{Deciding $K_{\ell+1}$-freeness in dense graphs}}\label{ALGO:K-freeness}
\begin{algorithmic}

\State \textbf{Input:}  
An $n$-vertex graph $G$ with $|G| \ge \mathrm{ex}(n,K_{\ell+1}) - k$. 

\State \textbf{Output:} 
"Yes", if $G$ is $K_{\ell+1}$-free; "No", otherwise.  
 
\State \textbf{Operations:}
            \begin{enumerate}
                \item Let $Z$ and $U$ be as defined in the proof of Theorem~\ref{THM:Embed-avg-Klique}. 
                If $z > \frac{12 \ell^2}{n}\left(k + \frac{\ell}{8}\right)$, then return "No"; otherwise, do the following operations. 
                \item Run Algorithm~\ref{ALGO:HammingClustering} with input $(G[U], \ell, \delta)$, where $\delta \coloneq \frac{1}{3\ell+1}$, and assume that $U_1 \cup \cdots \cup U_{\ell} = U$ is the output partition. 
                If $G[U_i] \neq \emptyset$ for some $i\in [\ell]$, then return "No"; otherwise, do the following operations. 
                \item 
                For every $v \in Z$ find the smallest index $i_v \in [\ell]$ such that $N_{G}(v) \cap U_{i_v} = \emptyset$. 
                If there is no such $i_v$ for some $v\in Z$, then return "No"; otherwise, do the following operations.  
                \item Let $V_j \coloneq U_j \cup \{v\in Z \colon i_v = j\}$ for $j\in [\ell]$. 
                If $\bigcup_{j\in [\ell]}G[V_j] = \emptyset$, then return "Yes"; otherwise, return "No". 
            \end{enumerate}
\end{algorithmic}
\end{algorithm}

\begin{proof}[Proof of Theorem~\ref{THM:Embed-avg-Klique}]
    Let $n \ge \ell \ge 2$ and $k \ge 1$ be integers satisfying $n \ge \max\{6\ell^2, 30 k \ell\}$. 
    Let $G$ be an $n$-vertex graph with at least $\mathrm{ex}(n,K_{\ell+1})-k$ edges. 
    Let $V_0\coloneqq V(G)$. 
    For $0 \le i \le n-1$, we pick a vertex $v_{i+1}$ of minimum degree in the induced subgraph $G[V_i]$, and let $V_{i+1}\coloneqq V_{i}\setminus \{v_{i+1}\}$.
    Let $z$ be the smallest positive integer $i$ such that $d_{G[V_i]}(v_{i+1}) > \frac{3\ell-4}{3\ell-1}(n-i)$ (if there is no such $i$, then let $z = n$).
    Let $Z\coloneqq \{v_i \colon i\in [z]\}$ and $U \coloneq V(G)\setminus Z$. 
    We claim that the following Algorithm~\ref{ALGO:K-freeness} can decide the $K_{\ell+1}$-freeness of $G$ in time $(\ell+4)n^2$.

    The validity of Algorithm~\ref{ALGO:K-freeness} will be established through the following claims.
    Suppose that $G$ is $K_{\ell+1}$-free. 
    First, notice from the fact $\mathrm{ex}(n,K_{\ell+1}) - \mathrm{ex}(n-1,K_{\ell+1}) = n - \lceil n/\ell\rceil$ that $\delta(G) \ge n - \lceil n/\ell\rceil - k > \frac{\ell-1}{\ell} n - k-1$.
    Since otherwise we would have $|G| 
        < n - \lceil n/\ell\rceil - k
        + \mathrm{ex}(n-1,K_{\ell+1})
        \le \mathrm{ex}(n,K_{\ell+1})-k$, 
    a contradiction. 
    
    \begin{claim}\label{CLAIM:size-Z}
        We have $z \le \frac{12 \ell^2}{n}\left(k + \frac{\ell}{8}\right)$. 
        In particular, $z \le \frac{13 n}{120 \ell}$. 
    \end{claim}
    \begin{proof}
        Since $d_{G[V_i]}(v_{i+1}) \le \frac{3\ell-4}{3\ell-1}(n-i)$ for $i \le z$, we obtain 
        \begin{align}\label{equ:GU-size-clique}
            |G[U]| 
            & \ge |G| - \sum_{i=0}^{z-1}\frac{3\ell-4}{3\ell-1}(n-i) \notag \\
            & \ge \mathrm{ex}(n,K_{\ell+1}) - k - \frac{3\ell-4}{2(3\ell-1)} (2n+1-z)z \notag \\
            & = \mathrm{ex}(n,K_{\ell+1}) - \frac{\ell-1}{2\ell}(2n-z)z - k + \frac{(2n-z)z}{2(3\ell^2-\ell)} -  \frac{3\ell-4}{2(3\ell-1)} z \notag\\
            & \ge \frac{\ell-1}{2\ell}(n-z)^2 - \left(k + \frac{\ell}{8} +  \frac{3\ell-4}{2(3\ell-1)} z - \frac{(2n-z)z}{2(3\ell^2-\ell)} \right),
        \end{align}
        where the last inequality follows from the fact that $\mathrm{ex}(n,K_{\ell+1}) = \frac{\ell-1}{2\ell}n^2 - \frac{s}{2}\left(1-\frac{s}{\ell}\right) \ge \frac{\ell-1}{2\ell}n^2 - \frac{\ell}{8}$, where $s\coloneqq n - \ell \lfloor n/\ell \rfloor$. 
        Let $k'\coloneq k + \frac{\ell}{8} +  \frac{3\ell-4}{2(3\ell-1)} z - \frac{(2n-z)z}{2(3\ell^2-\ell)} \le k + \frac{\ell}{8}$. 
        Since $|G[U]| \le \mathrm{ex}(n-z,K_{\ell+1}) \le \frac{\ell-1}{2\ell}(n-z)^2$, we have $k' \le 0$, and hence, 
        \begin{align*}
            z 
            \le \frac{4(3\ell^2-\ell)}{2n-z} \left(k+\frac{\ell}{8}\right)
            \le \frac{12\ell^2}{n} \left(k+\frac{\ell}{8}\right). 
        \end{align*}
        Here we used fact that $\frac{3\ell-4}{2(3\ell-1)} z \le \frac{(2n-z)z}{4(3\ell^2-\ell)}$, which follows from $n \ge 6\ell^2$. 
    \end{proof}
    Since $\delta(G[U]) > \frac{3\ell-4}{3\ell-1}(n-z)$, 
    it follows from Theorem~\ref{THM:AES} that there exists a partition $U_1 \cup \cdots \cup U_{\ell} = U$ such that $G[U_i] = \emptyset$ for all $i\in [\ell]$.
    Moreover, it follows from uniqueness that this partition is identical to the one generated by Algorithm~\ref{ALGO:K-freeness}.
    Let $x_i \coloneqq |U_i|/(n-z)$ for $i\in [\ell]$. 
    It follows from~\eqref{equ:size-xi} (by plugging in $\alpha = \frac{3\ell-4}{3\ell-1}$) that 
    \begin{align*}
        \frac{2}{3\ell-1} < x_i < \frac{3}{3\ell-1}
        \quad\text{for all}\quad i \in [\ell]. 
    \end{align*}
    For every $i\in [\ell]$ let $B_i 
                            \coloneqq \left\{v\in U_i \colon \exists u\in U\setminus U_i \text{ such that } \{v,u\}\not\in G\right\}$ and $U_i' \coloneq U_i \setminus B_i$. 
    Let $B \coloneqq B_1 \cup \cdots \cup B_{\ell}$, and notice that $|B| \le 2k'$ and $|B_i|\le k'$ for $i\in [\ell]$ (recall the definition of $k'$ from Claim~\ref{CLAIM:size-Z}). 
    Therefore, $|U_i'| \ge x_i(n-z) - k'$ for $i\in [\ell]$. 
    
    Observe that the induced subgraph of $G$ on $U'_1 \cup \cdots \cup U_{\ell}'$ is complete $\ell$-partite. 
    Therefore, for every $v\in Z$ there exists $i_{v}\in [\ell]$ such that $v$ has no neighbor in $U'_{i_v}$. 
    Suppose that $v$ has a neighbor $u\in B_{i_v}$. 
    Since $\delta(G[U]) > \frac{3\ell-4}{3\ell-1}(n-z)$, the vertex $u$ has at least 
    \begin{align*}
        \frac{3\ell-4}{3\ell-1}(n-z) - \sum_{k\in [\ell]\setminus\{i_v,j\}} x_k (n-z) - |B_j|
        & \ge \frac{3\ell-4}{3\ell-1}(n-z) - \frac{3(\ell-2)}{3\ell-1}(n-z) - k' \\
        & \ge \frac{2(n-z)}{3\ell-1} - k'
    \end{align*}
    neighbors in $U_j'$ for every $j \in [\ell]\setminus \{i_v\}$. 
    Since $G$ is $K_{\ell+1}$-free, there exists $j_v \in [\ell]\setminus \{i_v\}$ such that $v$ has no neighbor in $N_{G}(u) \cap U_{j_v}'$. 
    This means that the number of non-neighbors of $v$ in $U$ is at least 
    \begin{align*}
        x_{i_v}(n-z) - k' + \frac{2(n-z)}{3\ell-1} - k'
        & > \frac{4(n-z)}{3\ell-1} - 2k' \\
        & > \frac{4(n-z)}{3\ell-1} - 2\left(k + \frac{\ell}{8} +  \frac{3\ell-4}{2(3\ell-1)} z - \frac{(2n-z)z}{2(3\ell^2-\ell)}\right) \\
        & > \frac{4n}{3\ell-1} - 2\left(k + \frac{\ell}{8}\right), 
    \end{align*}
    where in the last inequality, we used $\frac{(2n-z)z}{2(3\ell^2-\ell)} \ge \frac{3\ell}{2(3\ell-1)} z$, which follows from $n \ge 6\ell^2$. 
    Since $n \ge \max\{30k\ell, 6\ell^2\}$, we have  $\frac{4n}{3\ell-1} - 2\left(k + \frac{\ell}{8}\right) >  \frac{n}{\ell}+k +1$. 
    Therefore, $d_{G}(v) < \frac{\ell-1}{\ell}n-k-1$, a contradiction.  
    This shows that $v$ does not have any neighbor in $B_{i_v}$, and hence, $v$ has no neighbor in $U_{i_v}$. 

    Let $Z_1 \cup \cdots \cup Z_{\ell} = Z$ be a partition such that for every $i\in [\ell]$ and for every $v\in Z_i$ we have  $N_{G}(v) \cap U_i = \emptyset$. 

    \begin{claim}\label{CLAIM:Klique-Zi-empty}
        We have $G[Z_i] = \emptyset$ for $i\in [\ell]$. 
    \end{claim}
    \begin{proof}
        First, we improve the lower bound $x_i > \frac{2}{3\ell-1}$ for $i\in [\ell]$. 
        Notice that $|G[U]| \le \sum_{1\le i < j \le \ell} x_ix_j (n-z)^2$. 
        So, it follows from~\eqref{equ:GU-size-clique} and the following inequality (which follows from the Maclaurin's inequality, see~{\cite[Lemma~2.2]{LMR23unif}}) 
        \begin{align*}
            \sum_{1\le i < j \le \ell} x_ix_j
            \le \frac{\ell-1}{2\ell} - \frac{1}{2}\sum_{i\in[\ell]}\left(x_i - \frac{1}{\ell}\right)^2
        \end{align*}
        that $\frac{1}{2}\sum_{i\in[\ell]}\left(x_i - \frac{1}{\ell}\right)^2
            \le \frac{k'}{(n-z)^2}
            \le \frac{k + \ell/8}{(n/2)^2}$. 
        It follows from $n \ge \max\{6\ell^2, 30 k \ell\}$ and $\ell \ge 2$ that 
        \begin{align*}
            x_i 
            \ge \frac{1}{\ell}
            -\left(\frac{8k + \ell}{n^2}\right)^{1/2}
            \ge \frac{1}{\ell} - \left(\frac{8}{30^2k}+\frac{1}{6^2\ell}\right)^{1/2}\cdot\frac{1}{\ell}
            \ge \frac{5}{6\ell}
            \quad\text{for all}\quad i\in [\ell]. 
        \end{align*}
        Suppose that this claim is not true, and by symmetry, we may assume that $\{v,v'\}\in G[Z_1]$ is an edge. 
        It follows from the $K_{\ell+1}$-freeness of $G$ that there exists $2\le i_{\ast} \le \ell$ such that $U_{i_{\ast}}' \cap N_{G}(v) \cap N_{G}(v') = \emptyset$. 
        By the Pigeonhole Principle, we may assume that at least half vertices in $U_{i_{\ast}}'$ are not adjacent to $v$. 
        This implies that 
        \begin{align*}
            d_{G}(v) 
            \le n - |U_1| - \frac{1}{2}|U_{i_{\ast}}'|
            & \le n - \frac{5}{6\ell}(n-z) - \frac{1}{2}\left(\frac{5}{6\ell}(n-z) - 2k'\right) \\
            & \le \frac{\ell-1}{\ell}n - k - 1 -\left(\frac{n}{4\ell} - k' -k - \frac{5z}{4\ell}-1\right).
        \end{align*}
        Since $k \le \frac{n}{30\ell}$, $\ell \le \frac{n}{6\ell}$,  $k' \le k + \ell/8 \le \frac{13 n}{240\ell}$ (all due to $n \ge \max\{6\ell^2, 30 k \ell\}$), and $z \le \frac{13 n}{120 \ell}$ (by Claim~\ref{CLAIM:size-Z}), we have $\frac{n}{4\ell} - k' -k - \frac{5z}{4\ell}-1 \le 0$. 
        Therefore, $d_{G}(v) \le \frac{\ell-1}{\ell}n - k - 1$, a contradiction. 
    \end{proof}
    Let $V_i \coloneq U_i \cup Z_i$ for $i\in [\ell]$. 
    It follows from Claim~\ref{CLAIM:Klique-Zi-empty} that $\bigcup_{i\in [\ell]}G[V_i] = \emptyset$, proving the correctness of Algorithm~\ref{ALGO:K-freeness}. 
\end{proof}


Now we present the proof of Theorem~\ref{THM:Embed-avg-Klique-lower-bound}. 
\begin{proof}[Proof of Theorem~\ref{THM:Embed-avg-Klique-lower-bound}]
    Fix $\delta>0$ (we may assume that $C \coloneq 1/\delta$ is an integer) and suppose to the contrary that there exists an algorithm $\mathcal{A}$ that solves $\textsc{Embed}_{\mathrm{avg}}$-$(K_{\ell+1},n,n^{1+\delta})$ in time $n^{o(\ell)}$.
    We claim that $\mathcal{A}$ can also solves $\textsc{Embed}$-$K_{\ell+1}$ in time $n^{o(\ell)}$, which would contradict the result by Chen--Huang--Kanj--Xia~\cite{CHKX06}. 

    The construction is very similar to that in the proof of Theorem~\ref{THM:Embed-min-Klique}. 
    Consider an arbitrary $n$-vertex graph $G$. 
    Let $\hat{G}$ be the graph obtained from $G$ by adding $\ell$ sets $V_1, \ldots, V_{\ell}$, each of size $n^{C}$, and adding new edges $\{u,v\}$ for all $(u,v) \in V_i \times V_j$ whenever $1 \le i < j \le \ell$.
    Let $N \coloneqq \ell n^{C}+ n$ denote the number of vertices in $\hat{G}$. 
    Observe that $K_{\ell+1} \subset G$ iff $K_{\ell+1} \subset \hat{G}$. 
    Since 
    \begin{align*}
        |\hat{G}|
        = \binom{\ell}{2}n^{2C} + |G| 
        \ge \frac{\ell-1}{\ell}\frac{(N-n)^2}{2}
        \ge \frac{\ell-1}{\ell}\frac{N^2}{2} - Nn
        > \mathrm{ex}(N,K_{\ell+1}) - N^{1+\delta}, 
    \end{align*}
    it follows from oue assumption that algorithm $\mathcal{A}$ can decide in time $N^{o(\ell)} = \left(\ell n^{C}+ n\right)^{o\left(\ell\right)} =  n^{o(\ell)}$ whether $K_{\ell+1} \subset \hat{G}$, and equivalently, whether $K_{\ell+1} \subset G$, proving our claim in the first paragraph.  
\end{proof}

\section{Concluding remarks}
Recall that the core of the algorithm for $\textsc{Embed}_{\mathrm{min}}$-$(K_{\ell+1},\alpha)$ when $\alpha > \frac{3\ell-4}{3\ell-1}$ is the structural theorem by Andr\'{a}sfai--Erd{\H o}s--S\'{o}s~\cite{AES74}. 
It seems worth exploring whether refined structural theorems (see e.g.~\cite{CJK97,GL11,ABGKM13,OS20}) can be used to push the lower bound for $\alpha$ further. 


The proof of Theorem~\ref{THM:Embed-min-F} can be easily modified to cover some hypergraph families with multiple extremal constructions (see e.g.~\cite{LM22,LMR23a}). 
There are hypergraph Tur\'{a}n problems whose structure of extremal constructions can exhibit a nonminal pattern (see e.g.~\cite{HLLMZ22}), a recursive pattern (see, e.g.~\cite{PI14}), or even a mixed recursive pattern (see, e.g.~\cite{LP22}). 
It is of interest to investigate whether Theorems~\ref{THM:Embed-min-F} and~\ref{THM:hypergraph-homomorphism} can be extended to cover nonminal/recursive/mixed recursive patterns.

Recall the nice structural characterization of minimal graphs: a graph is minimal iff it is complete. 
It would be interesting to explore a characterization of rigid graphs. 
Simple linear algebra arguments show that non-singular (i.e. the adjacency matrix is full rank) regular graphs are rigid.
We refer the reader to~\cite{Sci07} for related results on singular graphs. 
In general, one could ask for a characterization of the families of all minimal/rigid $r$-graphs.
\section*{Acknowledgement}
Theorem~\ref{THM:edge-stable-to-degree-stable} was previously presented by XL in a 2023 summer school at Suzhou University, and we would like to thank the organizers for their warm hospitality. 
XL would also like to thank Dhruv Mubayi and Christian Reiher for very inspiring discussions in a previous related project, and to Dhruv Mubayi for pointing out the application of degree-stability in spectral Tur\'{a}n problem. 
We also thank Yixiao Zhang for verifying the vertex-extendability of some hypergraphs in Table~\ref{tab:degree-stable.}. 
\bibliographystyle{abbrv}
\bibliography{FindingHAlgorithm}
\begin{appendix}
\section*{Definitions for hypergraphs in Table~\ref{tab:degree-stable.}}
\begin{itemize}
    \item A graph $F$ is \textbf{edge-critical} if there exists an edge $e\in F$ such that $\chi(F-e) < \chi(F)$. 
    \item Fix a graph $F$,  
        the \textbf{expansion} $H_{F}^{r}$ of $F$ is the $r$-graphs obtained from $F$ by adding a set of $r-2$ new vertices into each edge of $F$, and moreover, these new $(r-2)$-sets are pairwise disjoint. 
    \item Given an $r$-graph $F$ with $\ell+1$ vertices, 
        the \textbf{expansion} $H^{F}_{\ell+1}$ of $F$ is the $r$-graph obtained from $F$ by adding, for every pair $\{u,v\}\subset V(F)$ that is not contained in any edge of $F$, an $(r-2)$-set of new vertices, and moreover, these $(r-2)$-sets are pairwise disjoint. 
    \item We say a tree $T$ is an \textbf{Erd\H{o}s--S\'{o}s tree} if it satisfies the famous Erd\H{o}s--S\'{o}s conjecture on trees. 
    The \textbf{$(r-2)$-extension} $\mathrm{Ext}(T)$ of a tree $T$ is  
    \begin{align*}
        \mathrm{Ext}(T) := \left\{e\cup A \colon e \in T\right\}, 
    \end{align*}
    where $A$ is a set of $r-2$ new vertices that is disjoint from $V(T)$.   
    An $r$-graph $F$ is an \textbf{extended tree} if $F = \mathrm{Ext}(T)$ for some tree. 
    \item The ($r$-uniform) \textbf{generalized triangle} $\mathbb{T}_r$ is the $r$-graph with vertex set $[2r-1]$ and edge set  
    \begin{align*}
        \left\{\{1,\ldots,r-1, r\}, \{1,\ldots, r-1, r+1\}, \{r,r+1, \ldots, 2r-1\}\right\}. 
    \end{align*}
    \item Let $\mathcal{C}^{2r}_{3}$ (the expanded triangle) denote the $2r$-graph with vertex set $[3r]$ and edge set 
    \begin{align*}
        \left\{\{1,\ldots, r, r+1, \ldots, 2r\}, \{r+1, \ldots, 2r, 2r+1, \ldots, 3r\}, \{1,\ldots, r, 2r+1, \ldots, 3r\}\right\}. 
    \end{align*}
     \item The \textbf{Fano plane} $\mathbb{F}$ is the $3$-graph with vertex set $\{1,2,3,4,5,6,7\}$ and edge set
    \begin{align*}
        \{123,345,561,174,275,376,246\}. 
    \end{align*}
    \item Let $F_7$ ($4$-book with $3$-pages) denote the $3$-graph with vertex set $\{1,2,3,4,5,6,7\}$ and edge set 
    \begin{align*}
        \left\{1234, 1235, 1236, 1237, 4567\right\}. 
    \end{align*}
    \item Let $\mathbb{F}_{4,3}$ denote the $4$-graph with vertex set $\{1,2,3,4,5,6,7\}$ and edge set
    \begin{align*}
        \left\{1234, 1235, 1236, 1237, 4567\right\}. 
    \end{align*}
    \item Let $\mathbb{F}_{3,2}$ denote the $3$-graph with vertex set $\{1,2,3,4,5\}$ and edge set
    \begin{align*}
        \{123,124,125,345\}. 
    \end{align*}
    \item The $3$-graph $K_{4}^{3}\sqcup K_{3}^{3}$ has vertex set $\{1,2,3,4,5,6,7\}$ and edge set 
    \begin{align*}
        \{123,124,234,567\}. 
    \end{align*}
    \item The $r$-graph $M_{k}^{r}$ ($r$-uniform $k$-matching) is the $r$-graph consisting of $k$ pairwise disjoint edges. 
    \item The $r$-graph $L_{k}^{r}$ ($r$-uniform $k$-sunflower) is the $r$-graph consisting of $k$ edges $e_1, \ldots, e_k$ such that for all $1 \leq i < j \leq k$, it holds that $e_i \cap e_j = \{v\}$ for some fixed vertex $v$.
\end{itemize}
\end{appendix}
\end{document}